\documentclass[reqno]{amsart}

\usepackage{amsmath,amssymb}
\usepackage{graphicx,subfigure}
\usepackage{tikz}

\usepackage{color}

\newtheorem{theorem}{Theorem}[section]
\newtheorem{proposition}[theorem]{Proposition}
\newtheorem{corollary}[theorem]{Corollary}

\theoremstyle{definition}
 \newtheorem{definition}[theorem]{Definition}

\newtheorem*{definition*}{Definition}

\theoremstyle{remark}
\newtheorem{remark}[theorem]{Remark}

\numberwithin{equation}{section}

\newcommand{\pp}[2]{\frac{\partial#1}{\partial#2}}

\def\RR{\mathbb{R}}

\renewcommand\SS{\mathbb{S}}
\newcommand\minus\backslash

\newcommand\lan\langle
\newcommand\ran\rangle

\DeclareMathOperator\Div{div}

\renewcommand\leq\leqslant
\renewcommand\geq\geqslant
\newlength{\intwidth}

\title[Turing universality of Euler and a conjecture of Moore]{Turing universality of the incompressible Euler equations and a conjecture of Moore}

\author{Robert Cardona}\address{ Robert Cardona,
Laboratory of Geometry and Dynamical Systems, Department of Mathematics, Universitat Polit\`{e}cnica de Catalunya and BGSMath Barcelona Graduate School of
Mathematics,  Avinguda del Doctor Mara\~{n}on 44-50, 08028 , Barcelona  \it{e-mail: robert.cardona@upc.edu }
 }
 \thanks{Robert Cardona acknowledges financial support from the Spanish Ministry of Economy and Competitiveness, through the Mar\'ia de Maeztu Programme for Units of Excellence in R\& D (MDM-2014-0445) via an FPI grant.}
\author{Eva Miranda}\address{ Eva Miranda,
Laboratory of Geometry and Dynamical Systems $\&$ Institut de Matemàtiques de la UPC-BarcelonaTech (IMTech),  Universitat Polit\`{e}cnica de Catalunya, Avinguda del Doctor Mara\~{n}on 44-50, 08028 , Barcelona \\ $\&$ CRM Centre de Recerca Matem\`{a}tica
\\ $\&$ IMCCE, CNRS-UMR8028, Observatoire de Paris, PSL University, Sorbonne
Universit\'{e} \it{e-mail: eva.miranda@upc.edu }
 }
\thanks{Robert Cardona and Eva Miranda are partially supported by the grants MTM2015-69135-P/FEDER and PID2019-103849GB-I00 / AEI / 10.13039/501100011033, and AGAUR grant 2017SGR932. Eva Miranda is supported by the Catalan Institution for Research and Advanced Studies via an ICREA Academia Prize 2016.}
\author{Daniel Peralta-Salas} \address{Daniel Peralta-Salas, Instituto de Ciencias Matem\'aticas-ICMAT, C/ Nicol\'{a}s Cabrera, nº 13-15 Campus de Cantoblanco, Universidad Aut\'{o}noma de Madrid,
28049 Madrid, Spain \it{e-mail: dperalta@icmat.es} }
\thanks{Daniel Peralta-Salas is supported by the grants MTM PID2019-106715GB-C21 (MICINN) and Europa Excelencia EUR2019-103821 (MCIU). This work was partially supported by the ICMAT--Severo Ochoa grant CEX2019-000904-S}

\begin{document}
\begin{abstract}
In this article we construct a compact Riemannian manifold of high dimension on which the time dependent Euler equations are Turing complete. More precisely, the halting of any Turing machine with a given input is equivalent to a certain global solution of the Euler equations entering a certain open set in the space of divergence-free vector fields. In particular, this implies the undecidability of whether a solution to the Euler equations with an initial datum will reach a certain open set or not in the space of divergence-free fields. This result goes one step further in Tao's programme to study the blow-up problem for the Euler and Navier-Stokes equations using fluid computers. As a remarkable spin-off, our method of proof allows us to give a counterexample to a  conjecture of Moore dating back to 1998 on the non-existence of analytic maps on compact manifolds that are Turing complete.
\end{abstract}

\maketitle

\section{Introduction}\label{S:intro}

A Turing machine is a mathematical model of a theoretical device manipulating a set of symbols on a strip of tape with some specific rules.
Native to computer science, the notion of Turing completeness refers to a system that can simulate any Turing machine. The construction of dynamical systems (continuous or discrete) that are Turing complete is a classical problem that has received much attention in the last decades because of its deep connections with symbolic dynamics~\cite{Mo91,Mo,BC08}. Turing completeness has also been studied in several physical systems, from ray tracing problems in geometric optics~\cite{Reif} to quantum field theory~\cite{Free} or potential well dynamics~\cite{T1}.

Recently, Tao has proposed that the computational power of a Turing complete system could be used as a route to construct blow-up solutions to certain partial differential equations. Tao established such a Turing universality for nonlinear wave equations in~\cite{T1} and suggested that an analogous mechanism could be applied to the Navier-Stokes or the Euler equations in hydrodynamics~\cite{TNat} (to produce an initial datum that is programmed to evolve to a rescaled version of itself, as a Von Neumann self-replicating machine). The computational ability of a fluid flow was also suggested by Moore~\cite{Mo91} as a new manifestation of complexity in fluid mechanics associated to the undecidability of some fluid particle paths rather than to a chaotic behavior.

Motivated by Tao's programme and Moore's conjecture, in~\cite{CMPP2} we constructed a stationary fluid flow on a Riemannian $3$-manifold that can simulate a universal Turing machine. We also established other universality features of the steady solutions of hydrodynamics in~\cite{CMPP1} using high-dimensional Riemannian manifolds. Key to both results was the use of techniques in symplectic and contact geometry ranging from a simple path method argument to a sophisticated $h$-principle. Our goal in this paper is to go one step further in the study of the computational power of fluid dynamics by constructing time-dependent solutions of the Euler equations that are Turing complete.

 Since we shall deal with the Euler equations on Riemannian manifolds, let us briefly introduce them.

The dynamics of an incompressible fluid flow without viscosity on a Riemannian manifold $(M,g)$ is described by the Euler equations
\begin{align}
\frac{\partial X}{\partial t}+\nabla_XX=-\nabla P\,, \qquad \Div X=0\,,
\end{align}
with initial datum $X|_{t=0}=X_0$. The unknowns  are the velocity field $X$ of the fluid (a non autonomous vector field on $M$) and the hydrodynamical pressure $P$ (a time dependent scalar function on $M$). The symbol $\nabla_X$ denotes the covariant derivative along $X$ and $\Div$ is the divergence-operator, both computed using the Riemannian metric $g$. All along this paper $M$ is assumed to be compact, orientable and without boundary, and  solutions will be smooth ($C^\infty$).

The main theorem of this article shows that there exists a (constructible) Riemannian manifold $(M,g)$ whose associated Euler equations are capable of simulating any Turing machine. Roughly speaking, this means that the halting of any Turing machine with a given input is equivalent to a certain solution of the Euler equations defined for all time entering a certain open set in the space of divergence-free vector fields (for a precise definition, see Section~\ref{S:proof}). In the statement, the space of $C^\infty$ divergence-free vector fields on $(M,g)$ is denoted by $\mathfrak X^\infty_{vol}(M)$, and it is endowed with the Whitney $C^\infty$-topology.

\begin{theorem}\label{T.main}
There exists a (constructible) compact Riemannian manifold $(M,g)$ such that the Euler equations on $(M,g)$ are Turing complete.  In particular, the problem of determining whether a certain solution to the Euler equations with initial datum $X_0$ will reach a certain open set $U\subset\mathfrak X^\infty_{vol}(M)$ is undecidable.
\end{theorem}
\begin{remark}
The manifold $M$ is diffeomorphic to $SO(N)\times \mathbb T^{N}$ for some (large enough) integer $N$. The Riemannian metric $g$ is constructible but it has an involved expression. The dimension is very large and can be estimated as $\text{dim}(M)\lesssim 10^{35}$.
\end{remark}

A remarkable consequence of our result is the undecidability of the evolution of the Euler equations as an infinite dimensional dynamical system (as hard as the halting problem for Turing machines). This can be understood as a new complicated behavior of smooth solutions to the Euler equations, which complements other complex phenomena in fluid mechanics such as Lagrangian turbulence. Additionally, it raises the question of whether determining if an initial datum will blow-up or not under its Euler evolution is undecidable.

To finish this introduction we remark that a surprising spin-off of our method of proof yields a counterexample to a conjecture raised by Moore in 1998. Moore suggested~\cite{Mo} that no analytic function on a compact space can simulate a universal Turing machine with reasonable input and output encodings. A simple variation of the construction to prove Theorem~\ref{T.main} allows us to give a counterexample to this conjecture:

\begin{theorem}\label{T.Moore}
There exists a  Turing complete analytic diffeomorphism on the sphere $\mathbb{S}^{17}$.
\end{theorem}

\section{Turing machines and universality}\label{prelim}

In this section we briefly recall the concept of Turing machine and its connections with dynamics; in particular, we shall provide a precise definition of what we mean by a dynamical system being Turing complete (or Turing universal).

A Turing machine $T$ is defined by the following data:
\begin{itemize}
\item A finite set $Q$ of ``states'' including an initial state $q_0$ and a halting state $q_{halt}$.
\item A finite set $\Sigma$ which is the ``alphabet'' with cardinality at least two.
\item A transition function $\delta:(Q\times \Sigma) \longrightarrow (Q\times \Sigma \times \{-1,0,1\})$.
\end{itemize}

The evolution of a Turing machine is described as follows. Let us denote by $q\in Q$ the current state, and by $t=(t_n)_{n\in \mathbb{Z}}\in \Sigma^\mathbb{Z}$ the current tape. For a given Turing machine $(Q,q_0,q_{halt},\Sigma,\delta)$ and an input tape $s=(s_n)_{n\in \mathbb{Z}}\in \Sigma^{\mathbb{Z}}$ the machine runs applying the following algorithm:

\begin{enumerate}
\item Set the current state $q$ as the initial state and the current tape $t$ as the input tape.
\item If the current state is $q_{halt}$ then halt the algorithm and return $t$ as output. Otherwise compute $\delta(q,t_0)=(q',t_0',\varepsilon)$, with $\varepsilon \in \{-1,0,1\}$.
\item Replace $q$ with $q'$ and $t_0$ with $t_0'$.
\item Replace $t$ by the $\varepsilon$-shifted tape, then return to step $(2)$. Following Moore~\cite{Mo}, our convention is that $\varepsilon=1$ (resp. $\varepsilon=-1$) corresponds to the left shift (resp. the right shift).
\end{enumerate}

A Turing machine can be simulated by a dynamical system (a vector field or a diffeomorphism). Following~\cite{Mo91} we can define Turing completeness as:

\begin{definition}\label{TC}
Let $X$ be a vector field on a manifold $M$. We say it is Turing complete if for any integer $k\geq 0$, given a Turing machine $T$, an input tape $t$, and a finite string $(t_{-k}^*,...,t_k^*)$ of symbols of the alphabet, there exist an explicitly constructible point $p\in M$ and an open set $U\subset M$ such that the orbit of $X$ through $p$ intersects $U$ if and only if $T$ halts with an output tape whose positions $-k,...,k$ correspond to the symbols $t_{-k}^*,...,t_k^*$. A completely analogous definition holds for diffeomorphisms of $M$.
\end{definition}

We want to observe that in the construction we presented in~\cite{CMPP2}, the point $p$ depends on all the information, i.e., the Turing machine $T$, the input tape $t$ and the finite string $t^*=(t_{-k}^*,...,t_k^*)$, but the set $U$ is always the same (related to the halting state of a universal Turing machine). This is a technical difference with the construction introduced by Tao in~\cite{T1}, although they both have the same computational power. In Tao's notion of Turing completeness, the point $p$ depends only on the Turing machine $T$ and the input $t$; then, for any given finite string $t^*:=(t_{-k}^*,...,t_k^*)$ there is an open set $U_{t^*}$ such that the orbit through $p$ intersects $U_{t^*}$ if and only if $T$ halts with input $t$ and output whose positions $-k,...,k$ correspond to $t^*$. We shall also consider this dependence of the open set $U$ with $t^*$ in the constructions of Turing complete dynamics of the present article.

\begin{remark}
An important property of a Turing complete dynamical system is the existence of trajectories which exhibit undecidable long-term behavior. Specifically, it is undecidable to determine if the trajectory through an explicit point will intersect an explicit open set of the space. This follows from the undecidability of the halting problem for Turing machines.
\end{remark}

In the construction we present in Section~\ref{sec:polR} we make use of a special class of Turing machines, which is known to have the same computational power as a general Turing machine. Indeed, without any loss of generality, we may assume that the alphabet is $\Sigma=\{0,1,...,9\}$, where $0$ represents a special character referred to as the ``blank symbol". Additionally, we can also assume that a given tape of the machine has only a finite amount of symbols different from the blank symbol, i.e., for a (fixed) possibly large integer $k_0>0$, any tape is of the form
\begin{equation}\label{eq:string}
 ...00t_{-k_0}...t_{k_0}00...
\end{equation}
with $t_i\in\Sigma$. In particular, at any given step, there are only (at most) $2k_0+1$ non-blank symbols.

The space of configurations of the machine $T$ described above is of the form $Q\times A \subset Q\times \Sigma^{\mathbb{Z}}$, where $A$ is the subset of strings of the form~\eqref{eq:string}. A step of the algorithm is then represented by a global transition function
$$ \Delta: Q\times A \longrightarrow Q\times A\,,$$
where we set $\Delta(q_{halt},t):=(q_{halt},t)$ for any tape $t$.

\section{Turing complete polynomial vector fields in $\mathbb{R}^n$} \label{sec:polR}

In this section we construct a polynomial vector field on $\mathbb{R}^n$, for some (possibly) large $n$, which is Turing complete in the sense of Definition~\ref{TC}. Key to our construction is a result from~\cite{portu}, where it was shown that non-autonomous polynomial ODEs can simulate the transition function of a Turing machine.

Let $T=(Q,\Sigma, q_0,q_{halt},\delta)$ be any fixed Turing machine. We first recall the result in~\cite{portu} that allows us to simulate $T$ via a polynomial ODE. To this end, we need to encode each configuration $(q,t)\in Q\times \Sigma^{\mathbb{Z}}$ as a constructible point $x\in \mathbb{N}^3$. As mentioned in Section~\ref{prelim}, we may assume that $\Sigma=\{0,1,...,9\}$, where $0$ is the blank symbol. If $r$ denotes the cardinality of the space of states $Q$, we represent the elements of $Q$ by $\{1,...,r\}$. A given tape of the machine is of the form~\eqref{eq:string}. It is easy to encode such a tape in $\mathbb{N}^2$ by assigning to it the pair of natural numbers
\begin{align*}
y_1&=t_0+t_1\cdot 10+...+t_{k_0}\cdot 10^{k_0}\\
y_2&=t_{-1}+t_{-2}\cdot 10+...+t_{-k_0}\cdot 10^{k_0-1}
\end{align*}
The configuration $(q,t)$ is then represented by the point $x:=(y_1,y_2,q)\in \mathbb{N}^3$. Denote by $\phi$ the map that assigns to each configuration in $Q\times A$ its associated point in $\mathbb{N}^3$. The global transition function $\Delta$ can now be seen as a map from $\phi(Q\times A)\subset \mathbb{N}^3$ to $\phi(Q\times A)$. By extending this map as the identity on those points in $\mathbb{N}^3$ which are not in the image of $\phi$, we get a map from $\mathbb{N}^3$ to $\mathbb{N}^3$. To simplify the notation, we will still denote such a map by $\Delta:\mathbb{N}^3\longrightarrow \mathbb{N}^3$. (Observe that features like injectivity of $\Delta$ are not relevant here.)

The main Theorem in~\cite{portu} can then be stated as follows:

\begin{theorem}\label{thm:port}
Let $\Delta: \mathbb{N}^3 \to \mathbb{N}^3$ be the global transition function of a Turing machine $T$. Fix a constant $\varepsilon\in [0,\frac{1}{4}]$. There is a (constructible) polynomial $\widetilde p_T:\mathbb{R}^{m+4}\rightarrow \mathbb{R}^{m+3}$, for some $m\in \mathbb{N}$, and a (constructible) point $\widetilde y_0\in \mathbb{R}^m$ such that the ODE
$$ \frac{dz}{d\tau}=\widetilde p_T(\tau,z) $$
satisfies the following property. For every point $x_0\in \mathbb{N}^3\subset \RR^3$, the solution $z(\tau)$ to the ODE with initial condition $(x_0,\widetilde y_0)$ at $\tau=0$ satisfies
$$ |z_1(\tau)-\Delta^j(x_0)| < \varepsilon, $$
for all $\tau \in [j,j+\frac{1}{2}]$ and $j\in \mathbb{N}$, where $z\equiv(z_1,z_2)$ with $z_1\in \mathbb{R}^3$ and $z_2\in \mathbb{R}^m$.
\end{theorem}

The main result of this section is that, invoking this theorem, we can deduce that for any Turing machine, there is an autonomous polynomial vector field in some Euclidean space that simulates the machine in the same sense as Definition~\ref{TC}. That is, for every input $x_0$ and output $t^*$, there is a point and an open set such that the orbit of the vector field through $p$ intersects $U$ if and only if the machine $T$ halts with input $x_0$ and output whose positions $-k,...,k$ correspond to $t^*$.

\begin{proposition}\label{corTC}
Let $T$ be a Turing machine. There is a (constructible) polynomial vector field $p_T:\mathbb{R}^{m+4}\rightarrow \mathbb{R}^{m+4}$ and a (constructible) point $y_0\in \mathbb{R}^{m+1}$ such that the autonomous ODE
$$\frac{dZ}{d\tau}=p_T(Z)\,,$$
satisfies the following property. For any nonnegative integer $k\leq k_0$, take any finite substring $t^*=(t^*_{-k},...,t^*_k)\in \Sigma^{2k+1}$. There is an open set $U_{t^*} \subset \mathbb{R}^{m+4}$ such that for every $x_0\in \mathbb{N}^3\subset \RR^3$, the orbit of $Z(\tau)$ through $Z_0=(x_0,y_0)$ intersects $U_{t^*}$ if and only if the Turing machine $T$ with input $x_0$ halts with an output whose positions $-k,...,k$ correspond to $t^*_{-k},...,t^*_k$.
\end{proposition}

\begin{proof}
To simplify the exposition, we will say that a machine $T$ with some input halts with output $t^*$ if the output of the machine has in positions $-k,...,k$ the symbols $t^*_{-k},...,t^*_k$ where $t^*=(t^*_{-k},...,t^*_k)$.

For a given Turing machine $T$, let $\widetilde p_T:\mathbb{R}^{m+4}\rightarrow \mathbb{R}^{m+3}$ be the non-autonomous polynomial field given by Theorem~\ref{thm:port}. As explained above, the initial configuration of $T$ can be represented by a point $x_0\in\mathbb N^3\subset \RR^3$.

We define the (time-independent) polynomial vector field
$$ p_T(Z):=(1,\widetilde p_T(\omega,z)):\mathbb{R}^{m+4}\rightarrow \mathbb{R}^{m+4}\,, $$
where $Z:=(\omega,z)\in\RR\times \RR^{m+3}$. The associated polynomial ODE is
$$\frac{dZ}{d\tau}=(1,\tilde{p}_M(\omega,z))\,, $$
where $Z$ is now a coordinate in $\mathbb{R}^{m+4}$. Denote by $V$ the open set $V:=\bigcup_{i\in \mathbb{N}} (i,i+\delta)\subset \mathbb{R}$ for some (fixed) small $\delta>0$. For a given substring $t^*=(t^*_{-k},...,t^*_k)$, we claim that the open set
$$ U_{t^*}:=V \times U_\varepsilon^{t^*}\times \mathbb{R}^{m} \subset \RR^{m+4} $$
and the initial condition $Z_0=(0,x_0,\widetilde y_0)$ satisfy the required properties. Here $\widetilde y_0\in\RR^m$ is the point constructed in Theorem~\ref{thm:port}, and $U_\varepsilon^{t^*}\subset\RR^3$ is an $\varepsilon$-neighborhood of the set of (finitely many) points in $\RR^3$ associated to a configuration of $T$ of the form $(q_{halt},\overline t)$ with a tape $\overline t$  that has the symbols $t^*_{-k},...,t^*_k$ in positions $-k,...,k$ (all along this proof, $\varepsilon$ is any small enough constant).

Assume that $T$ halts with input $x_0$ and output $t^*$. Then there is some $j\in \mathbb{N}$ such that $\Delta^j(x_0)=(y_1,y_2,q)$ with $q=q_{halt}$, $y_1=t^*_0+t^*_1\cdot 10+...+t^*_k\cdot 10^k+ ...$ and
$y_2=t^*_{-1}+t^*_{-2}\cdot 10+...+t^*_{-k}\cdot 10^{k-1}+...$. By the properties of $\widetilde p_T$, the solution $Z(\tau)$ with initial datum $(0,x_0,\widetilde y_0)$ satisfies, for $\tau\in [j,j+\frac{1}{2}]$
$$|z_1(\tau)-\Delta^j(x_0)| < \varepsilon,$$
where $Z\equiv (\omega, z_1,z_2)$ with $\omega \in \mathbb{R}$, $z_1\in \mathbb{R}^3$ and $z_2\in \mathbb{R}^m$. Furthermore, since $\omega$ satisfies $$\frac{d \omega}{d\tau}=1\,,$$
we infer that for $\tau\in(j,j+\delta)$ the variable $\omega$ is in the interval $(j,j+\delta)$, which in turn implies that the orbit $Z(\tau)$ intersects $U_{t^*}$.

Conversely, assume that the orbit of the vector field defining our ODE with initial datum $(0,x_0,\widetilde y_0)$ intersects $U_{t^*}$. Equivalently, for the solution $Z(\tau)$ to the ODE with initial condition $Z(0)=(0,x_0,\widetilde y_0)$, there is a time $\hat \tau$ such that $Z(\hat \tau) \in U_{t^*}$. By construction of $U_{t^*}$, this means that there is a point $(y_1,y_2,q_{halt})\in \RR^3$ (uniquely defined by $U_{t^*}$) and some $j\in \mathbb{N}$ such that
\begin{equation*}
\begin{cases}
|z_1(\hat \tau)-(y_1,y_2,q_{halt})| < \varepsilon\,,\\
\omega(\hat \tau) \in (j,j+\delta)\,.
\end{cases}
\end{equation*}
By the equation $\frac{d\omega}{d\tau}=1$ we deduce that $\hat \tau\in (j,j+\delta)$. Therefore, by the properties of $z_1$ and $\widetilde p_T$, it easily follows that $z_1(\hat \tau)$ satisfies
$$|z_1(\hat \tau)-\Delta^{j}(x_0)| < \varepsilon\,. $$
We then deduce that $|(y_1,y_2,q_{halt})-\Delta^{j}(x_0)|< 2\varepsilon$. Since the only point representing a configuration of $T$ that lies in a $2\varepsilon$-neighborhood of $(y_1,y_2,q_{halt})$ is this very same point, we conclude that $\Delta^{j}(x_0)=(y_1,y_2,q_{halt})$. This shows that the machine $T$ halts with input $x_0$ and output $t^*$, which concludes the proof.
\end{proof}

Obviously, if we choose $T$ to be a universal Turing machine, we obtain from Proposition~\ref{corTC} a polynomial vector field in some Euclidean space which is Turing complete:

\begin{corollary}\label{corTC2}
There exists a (constructible) Turing complete polynomial vector field $P$ in $\RR^n$ provided that $n$ is large enough.
\end{corollary}

\begin{remark}\label{rem:dim}
As observed in \cite{Ha}, a universal Turing machine can be simulated via Theorem~\ref{thm:port} with $m+3=16$ variables and the degree $d$ of the polynomial $\widetilde p_N$ is equal to $56$. Hence Corollary~\ref{corTC2} yields a Turing complete polynomial vector field $P$ in $\mathbb{R}^{17}$ (from Proposition \ref{corTC}) of degree $56$.
\end{remark}

\section{Turing complete polynomial vector field on $\mathbb S^n$}\label{S:sphere}

In this section we construct a polynomial vector field on the $n$-dimensional sphere which is Turing complete. We remark that in~\cite{CMPP2} we constructed smooth ($C^\infty$) vector fields on $\mathbb S^3$ that are Turing complete. However, it is not obvious how to obtain polynomial vector fields (even analytic) from the aforementioned construction because it does not retain its computational power after an arbitrarily small perturbation (it is not possible to robustly simulate a universal Turing machine on a compact space~\cite{BGH}).

It is convenient to describe the $n$-dimensional sphere $\mathbb S^n$ as the unit sphere in $\RR^{n+1}$:
\[
\mathbb S^n:=\{x\in\RR^{n+1}: |x|=1\}\,.
\]
As usual, we say that a vector field $Y$ on $\mathbb S^n$ is polynomial if there exists a polynomial vector field $X$ on $\RR^{n+1}$ that is tangent to $\mathbb S^n$ and $X|_{\mathbb S^n}=Y$.

Taking as basis the Turing complete polynomial vector field we constructed in Section~\ref{sec:polR}, we now show that taking a suitable reparametrization of the field and the stereographic projection, it leads to a Turing complete polynomial vector field on $\mathbb S^n$.

\begin{theorem}\label{thm:sphereTC}
There exists a (constructible) polynomial vector field $Y$ of degree $58$ on the sphere $\SS^n$, $n\geq 17$, which is Turing complete.
\end{theorem}

\begin{proof}
Endowing $\RR^n$ with Cartesian coordinates $(x_1,...,x_n)$, the (inverse) stereographic projection $\varphi:\RR^n\to \mathbb \SS^n$ is defined as follows:
\[
y_0=\frac{r^2-1}{1+r^2}\,,\qquad y_k=\frac{2x_k}{1+r^2}\,,
\]
where $r^2:=x_1^2+\cdots+x_n^2$, and $y_0,...,y_n$ are coordinates in $\mathbb{R}^{n+1}$. It is immediate to check that $\mathbb S^n=\{y_0^2+y_1^2+\cdots+y_n^2=1\}\subset \mathbb{R}^{n+1}$.

Let $P$ be the (constructible) Turing complete polynomial vector field in $\mathbb{R}^n$ whose existence is established by Corollary~\ref{corTC2}, and denote by $d$ its degree. It is of the form
\[
P=\sum_{i=1}^n F_i \pp{}{x_i},
\]
where each $F_i\equiv F_i(x)$ is a polynomial of degree $d$ in the variables $x_1,...,x_n$. Let us compute $\varphi_*P$. By the chain rule we have:

\begin{align*}
\varphi_*\Big(F_i\pp{}{x_i}\Big)&=F_i\cdot\sum_{j=0}^n \pp{y_j}{x_i} \pp{}{y_j}\\
&=F_i\cdot \big[ (1-y_0)y_i\pp{}{y_0} +  (1-y_0-y_i^2)\pp{}{y_i} - \sum_{j\not\in \{0,i\}}  y_iy_j \pp{}{y_j}\big]\,,
\end{align*}
where $F_i$ is evaluated at $\Big(\frac{y_1}{1-y_0},...,\frac{y_n}{1-y_0}\Big)$. In particular, we deduce that
\begin{equation}\label{eq:ste}
\varphi_*P=\sum_{i=1}^n F_i.\big[ (1-y_0)y_i\pp{}{y_0} +  (1-y_0-y_i^2)\pp{}{y_i} - \sum_{j\neq \{0,i\}}  y_iy_j \pp{}{y_j}\big]\,.
\end{equation}
The vector field $\varphi_*P$ is then a rational field defined on $\RR^{n+1}$ except along the plane $\{y_0=1\}$. It is also easy to check that it is tangent to $\mathbb S^n$.

To define a global polynomial vector field on $\mathbb{R}^{n+1}$ tangent to $\mathbb S^n$, we will use a simple trick. Consider the vector field
\begin{equation}\label{eq:repa}
\widetilde P:=\frac{2^d}{(1+r^2)^d}P
\end{equation}
in $\mathbb{R}^n$. Notice that the integral curves of $\widetilde P$ and $P$ are the same, up to a reparametrization. In particular, since the proportionality factor is autonomous and positive, it is clear that $\tilde P$ is Turing complete if and only if $P$ is Turing complete. Indeed, the point $p$ and the open set $U$ associated to each Turing machine $T$, input and output are the same for $P$ and $\widetilde P$, the only difference being the time spent traveling the trajectories.

Accordingly, the (inverse) stereographic projection yields a vector field
\[
X:=\varphi_*\widetilde P=(1-y_0)^d\varphi_*P\,,
\]
whose expression in coordinates using Equation~\eqref{eq:ste} defines a polynomial vector field in $\mathbb{R}^{n+1}$ of degree $d+2$ that is tangent to the sphere $\mathbb S^n$. Therefore, the vector field
\[
Y:=X|_{\mathbb S^n}
\]
is a polynomial vector field on $\mathbb S^n$ by definition. Notice that the north pole $(1,0,\dots,0)$ of $\mathbb S^n$ is a zero of $Y$, so a trivial invariant set.

We claim that the vector field $Y$ on $\SS^n$ is also Turing complete. Indeed, let $T$ be a universal Turing machine, $t$ an input tape and $t^*=(t_{-k}^*,...,t_k^*)$ a finite string of the output. Since $\widetilde P$ in $\mathbb{R}^n$ is Turing complete, there is a constructible point $p\in \mathbb{R}^n$ and a constructible open set $U_{t^*}\subset \mathbb{R}^n$ such that $T$ halts with input $t$ and output $t^*$ if and only if the orbit of $\widetilde P$ through $p$ intersects $U_{t^*}$. Using that the (inverse) stereographic projection $\varphi$ is a diffeomorphism of $\RR^n$ onto $\mathbb S^n\backslash\{(1,0,\dots,0)\}$, and that the point $(1,0,\dots,0)$ is a zero point of $X$, this happens if and only if the orbit of $X$ through $\varphi(p)\in \SS^n$ intersects the open set $\varphi(U_{t^*})$. The open sets $U_{t^*}$ and $\varphi(U_{t^*})$ are schematically depicted in Figure \ref{fig:opensets}. Both $\varphi(p)$ and $\varphi(U_{t^*})$ are clearly constructible, since $U_{t^*},p$ and $\varphi$ are explicit.

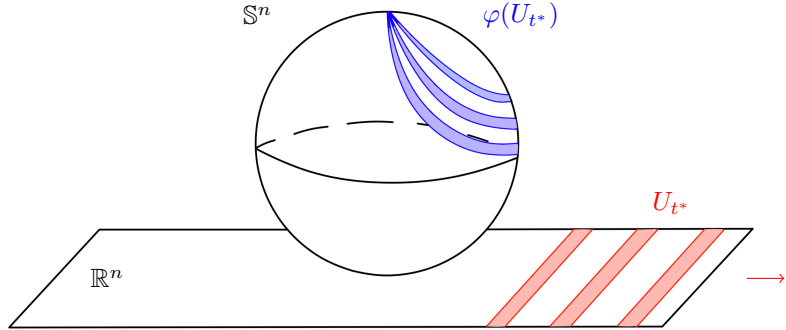
\begin{figure}[!h]
\begin{center}
\begin{tikzpicture}
     \node[anchor=south west,inner sep=0] at (0,0) {\includegraphics[scale=0.16]{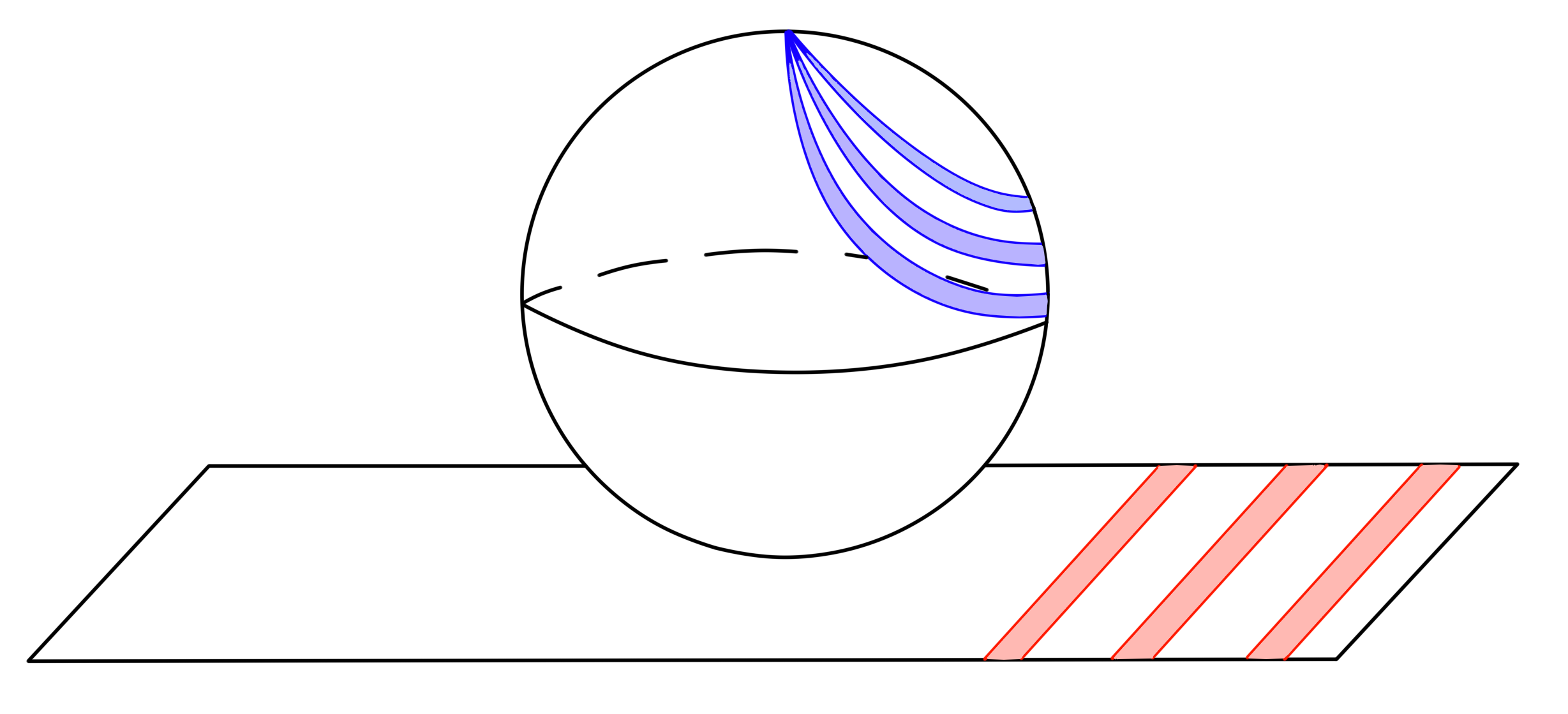}};

     \node[red] at (9,2) {$U_{t^*}$};
     \node[blue] at (7,4.5) {$\varphi(U_{t^*})$};
     \draw[red,->] (10,1)--(10.5,1);

     \node at (1.5,1) {$\mathbb{R}^n$};
     \node at (3.5,4.5) {$\mathbb{S}^n$};
\end{tikzpicture}
\caption{Open sets lifted to the sphere}
\label{fig:opensets}
\end{center}
\end{figure}

The theorem then follows from Remark~\ref{rem:dim}, which allows us to take dimension $n=17$ and a polynomial vector field $P$ of degree $56$, which leads to a polynomial field $X$ of degree $58$ after the (inverse) stereographic projection.
\end{proof}

\section{Proof of the main theorem}\label{S:proof}

We are now ready to prove Theorem~\ref{T.main}. First, by analogy with Definition~\ref{TC}, let us introduce the concept of Turing completeness of the Euler equations (on a Riemannian manifold $(M,g)$) as an infinite dimensional dynamical system.

\begin{definition}
The Euler equations on $(M,g)$ are Turing complete if the following property is satisfied. For any integer $k\geq 0$, given a Turing machine $T$, an input tape $t$, and a finite string $(t_{-k}^*,...,t_k^*)$ of symbols of the alphabet, there exist an explicitly constructible vector field
$X_0\in \mathfrak X^\infty_{vol}(M)$ and a constructible open set $U\subset \mathfrak X^\infty_{vol}(M)$ such that the solution to the Euler equations with initial datum $X_0$ is smooth for all time and intersects $U$ if and only if $T$ halts with an output tape whose positions $-k,...,k$ correspond to the symbols $t_{-k}^*,...,t_k^*$.
\end{definition}

Key to the proof of Theorem~\ref{T.main} is a remarkable embedding theorem established in~\cite{TL} which allows one to embed a generic finite dimensional dynamics into the (infinite dimensional) Euler flow on some compact manifold. More precisely, following~\cite{TL}, given a smooth vector field $Y$ on a compact manifold $N$, we say that $(N,Y)$ can be embedded into the Euler equations for a compact Riemannian manifold $(M,g)$ if there exists a (constructible) injective smooth immersion $\Phi:N\to \mathfrak X^\infty_{vol}(M)$ which maps the integral curves of $Y$ to solutions of the Euler equations on the invariant submanifold $\Phi(N)$. Specifically, for any integral curve $\phi_t:\RR\to N$ defined by the vector field $Y$, the path in $\mathfrak X^\infty_{vol}(M)$ defined by $X_t:=\Phi\circ \phi_t: \RR\to \mathfrak X^\infty_{vol}(M)$ is a smooth solution to the Euler equations on $(M,g)$ for some pressure $P:\RR \to C^\infty(M)$. For our purposes, it is enough to invoke the following result:

\begin{theorem}[Theorem 1.3 in~\cite{TL}]\label{thm:Eulemb}
Let $Y$ be a polynomial vector field on $\mathbb{S}^n$. Then there is a (constructible) compact Riemannian manifold $(M,g)$ such that $(\mathbb S^n,Y)$ can be embedded into the Euler equations for $(M,g)$.
\end{theorem}
\begin{remark}
The manifold $M$ is of the form $SO(N)\times \mathbb T^{N}$, and the metric $g$ has a cumbersome expression but it is constructible. If the polynomial vector $Y$ on $\mathbb{S}^n$ has degree $d$, the integer $N\equiv N(n,d)$ can be computed as (provided that $n\geq 2$)
\[
N(n,d)=\sum_{j=0}^{d+1} \binom{n-1+j}{j}\cdot \frac{2j+n-1}{j+n-1}\,,
\]
see Section~2.3 in~\cite{TL}.
\end{remark}

\begin{proof}[Proof of Theorem~\ref{T.main}]

By Theorem~\ref{thm:sphereTC}, there is a (constructible) Turing complete polynomial vector field $Y$ on $\mathbb{S}^{17}$ of degree $58$. Denote the associated flow by $\phi_t$. Applying the embedding Theorem~\ref{thm:Eulemb}, we can construct (explicitly) a compact Riemannian manifold $(M,g)$ and an embedding $\Phi:\mathbb S^{17}\to \mathfrak X^\infty_{vol}(M)$ such that $\Phi(\mathbb S^{17})$ is an invariant submanifold, and $X_t:=\Phi\circ \phi_t$ is the unique (smooth) solution to the Euler equations with initial datum $X_0=\Phi\circ \phi_0$.

Now, in view of Theorem~\ref{thm:sphereTC}, given a Turing machine $T$, an input tape $t$ and some output $t^*$, there is a point $p \in \mathbb{S}^{17}$ and an open set $U_{t^*}\subset \mathbb{S}^{17}$ such that the orbit of $Y$ through $p$ intersects $U_{t^*}$ if and only if $T$ halts with input $t$ and output $t^*$. By construction of the embedding $\Phi$, this is satisfied if and only if the solution to the Euler equations with initial datum $X_0=\Phi(p)$ intersects the set $\Phi(U_{t^*})\subset \mathfrak{X}^{\infty}(M)$. Since $\Phi(\mathbb{S}^{17})$ is invariant by the Euler flow, this then happens if and only if the (unique) smooth solution to the Euler equations with initial datum $X_0$ intersects an open neighborhood $V_{t^*}\subset \mathfrak{X}^{\infty}(M)$ of $U_{t^*}$ (in the $C^\infty$ topology), such that $V_{t^*}\cap \Phi(\mathbb{S}^n)=U_{t^*}$. Here we have used that the embedding $\Phi$ constructed in~\cite{TL} is $C^\infty$, and hence, for any $q\in \mathbb S^{17}$, the intersection of an open neighborhood of the point $\Phi(q)\in \mathfrak{X}^{\infty}(M)$ with $\Phi(\mathbb{S}^{17})$ is diffeomorphic to an open neighborhood of $q$ in $\mathbb S^{17}$. This completes the proof of the Turing universality of the Euler equations for $(M,g)$.
\end{proof}

\section{Final remark: Moore's conjecture}\label{S:Moore}

An unexpected spin-off of our construction of a Turing complete polynomial vector field on $\mathbb S^{17}$, cf. Theorem~\ref{thm:sphereTC}, is that it allows us to disprove a conjecture stated by Moore. Specifically, in~\cite{Mo} Moore conjectured that a universal Turing machine cannot be simulated by an analytic function on a compact space (with reasonable input and output encodings).

The main idea to prove Theorem~\ref{T.Moore} is to show that the $\delta$-time flow map (which is analytic) of the reparametrized vector field $\widetilde P$ in Equation~\eqref{eq:repa} is Turing complete for some $\delta>0$. Then, using the stereographic projection, we will obtain the desired Turing complete analytic diffeomorphism on $\mathbb S^{17}$. In the proof, we use the same notation introduced in Section~\ref{sec:polR} without further mention.

\begin{proof}[Proof of Theorem~\ref{T.Moore}]

Given a universal Turing machine $T$, let us consider the polynomial vector field $P$ in $\mathbb{R}^{m+4}$ that we constructed in Corollary~\ref{corTC2}. Its associated ODE is
\begin{equation*}
\frac{dZ}{d\tau}=(1,\widetilde p_T(\omega,z))\,,
\end{equation*}
where the polynomial $\widetilde p_T$ appears in Theorem~\ref{thm:port} and $Z=(\omega,z)$. Now we change its parametrization to define
$\widetilde P:= \frac{1}{(r^2+1)^d} P\,,$
where $r^2=\omega^2+|z|^2$. We denote by $\phi_{\tau}$ the flow of $\widetilde P$ (which is global because $\widetilde P$ is a bounded vector field). Since $\widetilde P$ is an analytic field, it is well known that its flow is analytic as well. Fixing a constant $\delta\in(0,\frac{1}{2})$, we claim that the $\delta$-time flow $F:=\phi_\delta$ of $\widetilde P$ is a Turing complete diffeomorphism of $\RR^{m+4}$.

For a given input $x_0$ of the Turing machine $T$, we will use the same initial point that we constructed in Proposition~\ref{corTC}. The initial point is then $p=(0,x_0,\widetilde y_0)$, and we take the open set $U_{t^*}= V \times U_\varepsilon^{t^*}\times \mathbb{R}^{m} \subset \RR^{m+4}$ where $V=\bigcup_{i\in \mathbb{N}} (i,i+1/2)$ and $U_\varepsilon^{t^*}\subset\RR^3$ is an $\varepsilon$-neighborhood of the set of (finitely many) points in $\RR^3$ associated to a configuration of $T$ of the form $(q_{halt},\overline t)$ with a tape $\overline t$  that has the symbols $t^*_{-k},...,t^*_k$ in positions $-k,...,k$.

First, assume that $T$ halts with input $x_0$ and output $t^*$ at step $j$. By construction, then $\Delta^r(x_0)=\Delta^{j}(x_0)=(y_1,y_2,q_{halt})$ for all integers $r\geq j$, and the tape associated to $(y_1,y_2)$ coincides with $t^*$ in the positions $-k,...,k$. The ODE associated to $\widetilde P$ is
\begin{equation*}
\frac{d\widetilde Z}{d\tau}=(f,f\cdot\widetilde p_T(\widetilde\omega,\widetilde z))\,,
\end{equation*}
where $f:=(1+\widetilde\omega^2+|\widetilde z|^2)^{-d}$. It is easy to check that the solution $\widetilde Z(\tau)=(\widetilde \omega(\tau),\widetilde z_1(\tau), \widetilde z_2(\tau))$ of this ODE with initial condition $p$ satisfies
\begin{equation}\label{eq:rep}
\widetilde Z(\tau)=Z(\widetilde \omega(\tau))\,,
\end{equation}
where $Z(\tau)=(\omega(\tau),z_1(\tau),z_2(\tau))$ denotes the solution to the ODE associated to $P$ and same initial condition.

The properties of $z_1(\tau)$ ensure that
$$ |z_1(\tau)-\Delta^i(x_0)|<\varepsilon $$
for $\tau \in [i,i+\frac{1}{2}]$ and $i\in\mathbb N$. By Equation~\eqref{eq:rep} this implies that
\begin{equation}\label{eq:repTC}
|\widetilde z_1(\tau)- \Delta^i(x_0)|<\varepsilon
\end{equation}
for $\widetilde\omega(\tau)\in [i,i+\frac{1}{2}]$. On the other hand, since the reparametrization factor $f$ is strictly smaller than $1$ (except at the origin), we infer that $\tilde \omega(\tau) \in [i,i+\frac{1}{2}]$ if $\tau \in [t_i,t_i+A_i]$ for some $t_i>i$ and $A_i>\frac{1}{2}$. Since the machine halts with final configuration $\Delta^j(x_0)$, we deduce that for $\tau \in [t_j,t_j+A_j]$ the solution $\widetilde Z(\tau)$ satisfies
$$  |\widetilde z_1(\tau)- \Delta^j (x_0)|<\varepsilon\,. $$
Being the interval $[t_j,t_j+A_j]$ of size greater than $1/2$, there is a natural number $r\in \mathbb{N}$ such that $r\delta \in (t_j,t_j+A_j)$. Then $F^r(p)=\phi_{\delta}^r(p)=\phi_{r\delta}(p)=:(\omega_r, z_1^r,z_2^r)$ satisfies that $z_1^r=\widetilde z_1(r\delta)$ is in the $\varepsilon$-neighborhood of $\Delta^j(x_0)$. Moreover, we also conclude that $\widetilde\omega(r\delta)=\omega_r \in (j,j+1/2)$ because $r\delta \in (t_j,t_j+A_j)$, thus implying that $F^r(p)\in U_{t^*}$ as claimed.

To check the converse implication, assume that there is a natural number $r\in \mathbb{N}$ such that $F^r(p)=:(\omega_r,z_1^r,z_2^r)\in U_{t^*}$. This is equivalent, by definition, to the assumption that the solution $\widetilde Z(\tau)=(\widetilde \omega(\tau),\widetilde z_1(\tau),\widetilde z_2(\tau))$ satisfies $\tilde Z(r\delta)\in U_{t^*}$. In particular $\widetilde \omega(\tau)\in (j,j+1/2)$ for some $j\in \mathbb{N}$. Equation~\eqref{eq:repTC} implies that
$$ |\widetilde z_1(r\delta)- \Delta^j(x_0)|<\varepsilon\,,$$
where $\widetilde z_1(r\delta)=:z_1^r$. Moreover, by assumption, $|z_1^r -(y_1,y_2,q_{halt})|<\varepsilon$ for some configuration $(y_1,y_2,q_{halt})$ whose associated tape coincides with $t^*$ in the position $-k,...,k$. We then deduce that $|\Delta^j(x_0)-(y_1,y_2,q_{halt})|<2\varepsilon$, and therefore, since any two points representing a configuration of the machine are at distance at least $1$, this yields $\Delta^j(x_0)=(y_1,y_2,q_{halt})$. Accordingly, the machine $T$ halts with final configuration $(y_1,y_2,q_{halt})$, and hence with output $t^*$.

Summarizing, we have established that the $\delta$-time flow of the vector field $\widetilde P$ is a Turing complete diffeomorphism of $\RR^n$, $n:=m+4$. As shown in Section~\ref{S:sphere}, $\widetilde P$ can be lifted to a polynomial vector field $Y$ on the sphere $\mathbb{S}^{n}$ via the stereographic projection, and the north pole $N_0$ is a zero point of $Y$. Denoting the (inverse) stereographic projection as $\varphi:\mathbb{R}^n\rightarrow \mathbb S^n$, and by $\rho_t$ the flow of $Y$, we infer that $\rho_t(N_0)=N_0$ for all $t\in\RR$ ($N_0$ is a fixed point), and in the complement $\mathbb S^n\backslash\{N_0\}$, $\rho_t$ is conjugate to the flow $\phi_t$ defined by $\widetilde P$, i.e.,
$$\rho_t= \varphi \circ \phi_t \circ \varphi^{-1}\,, $$
for all $t\in\RR$. It is then easy to check that the map $\Pi:=\rho_{\delta}:\mathbb S^n\to \mathbb S^n$ is a Turing complete analytic diffeomorphism on $\mathbb{S}^n$. Indeed, as shown above, for a given input $x_0$ of the machine and an output $t^*=(t_{-k}^*,...,t_k^*)$, there is a point $p\in\RR^n$ and open set $U_{t^*}\subset\RR^n$, such that the iterates of $\phi_{\delta}$ through $p$ reach $U_{t^*}$ if and only if the machine $T$ halts with the aforementioned output. Therefore, after the stereographic projection, this property is satisfied if and only if the iterates of $\rho_{\delta}$ through the point $\varphi(p)$ reach the open set $\varphi(U_{t^*})\subset \mathbb S^n\backslash\{N_0\}$. We conclude that the diffeomorphism $\rho_{\delta}$, which is analytic because $Y$ is an analytic field, is Turing complete. Finally, as in Theorem~\ref{thm:sphereTC}, we can take the dimension $n=17$, which completes the proof of the theorem.
\end{proof}

\bibliographystyle{amsplain}

\end{document}